%% file: Shotgun-Linial_meshulam.tex
\documentclass{amsart}

\usepackage{amssymb}
\usepackage{amsmath}
\usepackage{amsthm}
\usepackage{amsbsy}
\usepackage{bm}
\usepackage{hyperref}
\date{\today}
\usepackage{cite}
\usepackage{array}
\usepackage{graphicx}
\usepackage{float}
\usepackage{upgreek}
\usepackage{color}
\usepackage{csquotes}

\input{mypreamble}
\title[Shotgun Assembly of random graphs]{Shotgun Assembly of  Linial-Meshulam model}

\author{Kartick Adhikari}
\address{Department of Mathematics, Indian Institute of Science Education and Research, Bhopal 462066}

\email{kartickmath [at] gmail.com}

\author{Sukrit Chakraborty}
\address{Stat-Math Unit, Indian Statistical Institute, 203 B T Road, Kolkata 700108}
\email{sukrit049 [at] gmail.com}

\date{\today}

\begin{document}

\newcommand{\acr}{\newline\indent}

	\begin{abstract}
		In a recent paper \cite{gaudio2020shotgun}, J. Gaudio and E. Mossel  studied the shotgun assembly of the Erd\H os-R\'enyi graph $\mathcal G(n,p_n)$ with $p_n=n^{-\alpha}$, and  showed that the graph is reconstructable form its $1$-neighbourhoods if $0<\alpha < 1/3$ and not reconstructable  from its $1$-neighbourhoods if $1/2 <\alpha<1$. In this article, we generalise the notion of reconstruction of graphs to the reconstruction of simplicial complexes. We show that the Linial-Meshulam model $Y_{d}(n,p_n)$ on $n$ vertices with $p_n=n^{-\alpha}$ is reconstructable from its $1$-neighbourhoods when $0< \alpha < 1/3$ and is not reconstructable form its $1$-neighbourhoods when $1/2 < \alpha < 1$.
	\end{abstract}
	\maketitle
	
	\section{Introduction}
	In \cite{mossel2015shotgun}, Mossel and Ross introduced the shotgun assembly of graphs. The shotgun assembly of a graph means reconstructing the graph from a collection of vertex neighbourhoods.  
	 The motivation  comes from DNA shotgun assembly (determining a DNA sequence from multiple short nucleobase chains), reconstruction of neural networks (reconstructing a big neural network from subnetworks), and the random jigsaw puzzle problem. See \cite{MR3969756} and references therein.  
	
	The recent development of random jigsaw puzzles can be found in \cite{bordenave2020shotgun} and \cite{martinsson2016shotgun}. The graph shotgun assembly for various models was studied extensively. For examples, the random regular graphs and the labelled graphs were considered in \cite{mossel2015shotgun} and \cite{MR3969756}, respectively. The reconstruction of the {E}rd{\H{o}}s--{R}{\'e}nyi  graph is well studied.
	% in \cite{gaudio2020shotgun, huang2021shotgun, MR3969756}.  
	
	The {E}rd{\H{o}}s--{R}{\'e}nyi  graph  \cite{ergraph} \cite{erdos1960evolution}, denoted by $\GG(n,p_n)$, is a random graph on $n$ vertices, where each edge is added independently with probability $p_n\in [0,1]$. In \cite{gaudio2020shotgun}, Gaudio and Mossel showed that $\GG(n,p_n)$ with $p_n=n^{-\alpha}$ is reconstructable if $0<\alpha<1/3$ and not reconstructable if $1/2<\alpha<1$ from its $1$-neighbourhoods.  Later, Huang and Tikhomirov showed that 
	$\mathcal G(n,p_n)$ with $p_n=n^{-\alpha}$ is reconstructable if $0<\alpha<1/2$ and not reconstructable if $1/2<\alpha<1$ from its $1$-neighbourhoods \cite{huang2021shotgun}. The reconstruction of  $\GG(n,p_n)$ from its $3$ and $2$-neighbourhoods are considered in \cite[Theorem 4.5]{MR3969756}  and \cite[Theorem 4]{gaudio2020shotgun} respectively.

	In this article, by generalising the notion of graph shotgun assembly, we introduce the notion of shotgun assembly of simplicial complexes. See Section \ref{sec:preli}. The problem of shotgun assembly essentially tells us whether the local structure contains all the information about its global structure. Here we only consider the shotgun assembly problem for the Linial-Meshulam model. However, our notion of  shotgun assembly of simplicial complexes can potentially be used for other simplicial complexes, for example, the multi-parameter random simplicial complexes \cite{CF2016book, FCF2019}.

	The Linial-Meshulam model, denoted by $Y_d(n,p_n)$, is a random $d$-dimensional simplicial complex on $n$ vertices with a complete $(d-1)$-skeleton, in
	which each $d$-dimensional simplex is added independently with probability $p_n\in [0,1]$. See Section \ref{sec:preli} for details. In \cite{linialmeshulam}, Linial and Meshulam introduced this model for $d=2$. Later, it was extended for $d\ge 3$ by Meshulam and Wallach \cite{meshulamwallach}. After that this model has been studied extensively, for example see \cite{HJ2013, GW2016, LP2016, KR2017, HS2017, PR2017, LP2019, HK2019, LP2022}.	Observe that  $Y_1(n,p_n)=\GG(n,p_n)$, in other words, $d=1$ gives the {E}rd{\H{o}}s--{R}{\'e}nyi graph.

	 We show that $Y_d(n,p_n)$ for any $d \in \mathbb{N}$ with $p_n=n^{-\alpha}$ is reconstructable if $0<\alpha<1/3$ and not reconstructable if $1/2<\alpha<1$ from its $1$-neighbourhoods.  See Theorems \ref{main.thm.1} and \ref{main.thm.2}. The  meaning of reconstruction of a simplicial complex  from its $1$-neighbourhoods is given in Section \ref{sec:preli}.  We believe that the range $0<\alpha<1/3$ of reconstruction is not optimal, the optimal range should be $0<\alpha<1/2$. See Conjecture \ref{conj}.
	
	The rest of the article is organized as follows.  In Section \ref{sec:preli} we introduce the definition of the reconstruction of simplicial complexes, and relevant notation. The main two results are stated in Section \ref{sec.mainres}. The proofs of Theorems \ref{main.thm.1} and \ref{main.thm.2} are given in Sections \ref{sec.pfmain1} and \ref{sec.pfmain2} respectively.

\section{Preliminaries}\label{sec:preli}
 Let $X_0$ be a finite set. A {\it finite abstract simplicial complex} $X$ on $X_0$ is a collection of subsets $S\subset X_0$ satisfying the following property
\[
 T\subset S  \mbox{ and } S\in X \implies T\in X.
\]
For example, $X=\{\emptyset, \{1\},\{2\},\{3\},\{1,2\},\{2,3\},\{1,3\},\{1,2,3\}\}$ is an abstract simplicial complex on $\{1,2,3\}$. We call a set $S\in X$ with $|S|=k+1$ as a $k$-dimensional simplex. In particular, we call a vertex as a `$0$-simplex', an edge as a $1$-simplex, a triangle as a $2$-simplex, and so on. Also a convention that $\dim(\emptyset)=-1$. For ease of writing, we write {\it complex} instead of abstract simplicial complex in the rest of the article.  The maximum of the dimensions of all simplexes in $X$ is called the dimension of complex $X$, denoted by $\dim(X)$. That is,
\[
\dim(X):=\max\{\dim(S)\suchthat S \in X\}.
\]
Observe that if $\dim(X)=1$ then $X$ can be viewed as a graph.

 For $0\le j\le \dim(X)$, the set of all $j$-dimensional simplexes of $X$ is denoted by 
$$
X^j:=\{\sigma\in X\suchthat \dim(\sigma)=j\}.
$$ 
We say  $\sigma, \sigma'\in X^j$ are neighbour if $\sigma\cup \sigma'\in X^{j+1}$. Then we write  $\sigma\sim \sigma'$.  A similar notion was introduced in \cite{PR17}. 
We say the distance of $\sigma,\sigma'\in X^j$ is $k\in \N\cup\{0\}$ if $k$ is the least possible number such that there exist $\sigma_0,\sigma_1\ldots, \sigma_k\in X^j$ with $\sigma=\sigma_0$ and $\sigma_k=\sigma'$ such that $\sigma_i\sim \sigma_{i+1}$ for $0\le i\le k-1$. Then we write $\dist(\sigma, \sigma')=k$.
Define 
\[
X_{\sigma,k}:=\{\sigma'\in X^j\suchthat \dist(\sigma,\sigma')\le k\},
\]
the set of all $j$-simplexes which are within distance $k$ from $\sigma$. Clearly $\sigma\in X_{\sigma,k}$ for all $k\ge 0$. Note that  if $k=0$ or $\dim(\sigma)=\dim(X)$ then $X_{\sigma,k}=\{\sigma\}$, as there is no $\sigma'(\neq \sigma)\in X$ such that  $\sigma'\sim \sigma$. Thus $k=0$ and $\dim(\sigma)=\dim(X)$ are two trivial cases.

Let $k\ge 1$ and $j<\dim(X)$. The {\it $k$-neighbourhood} of $\sigma\in X^j$ is the $(j+1)$-dimensional sub-complex induced by $X_{\sigma,k}$, denoted by $N_{k,X}(\sigma)$. That is, 
\[
N_{k,X}(\sigma):=\{\tau\in X\suchthat \tau \subseteq \sigma'\cup \sigma'' \mbox{ for some } \sigma',\sigma''\in X_{\sigma,k}\}.
\]
In particular, if $\dim(X)=1$ and $v\in X_0$ then $N_{1,X}(v)$ refers to the sub-graph induced by $v$ and its neighbours $\{w\in X_0\suchthat v\sim w\}$.

We say two complexes $X$ and $Y$ (on $X_0$ and $Y_0$ respectively) are {\it  isomorphic} (denoted by $X\simeq Y$) if there exists a bijective function $f: X_0 \to Y_0$ such that 
\[
\{\sigma^0,\sigma^1,\ldots,\sigma^k\}\in X \iff \{f(\sigma^0),\ldots, f(\sigma^k)\}\in Y, \mbox{ for $0\le k\le \dim(X)$}. 
\]
It is clear that if $X\simeq Y$ then $|X_0|=|Y_0|$ and $\dim(X)=\dim(Y)$.   If $\dim(X)=1$ then  $X\simeq Y$  means the  two graphs $X$ and $Y$ are isomorphic.

We say two complexes $X$ and $\widetilde{X}$ on $X_0$ have same $k$-neighbourhoods  if  
\begin{align}\label{eqn:k-neighbourh}
N_{k,X}(\sigma)\simeq N_{k  ,\widetilde{X}}(\sigma) \mbox{ for all $\sigma\in X$,}
\end{align}
that is, the $k$-neighbourhoods of all simplexes in both complexes are isomorphic. In this case we write  $X\simeq_k \widetilde{X}$. Observe that if $\dim(\sigma)=\dim(X)$ then \eqref{eqn:k-neighbourh} holds trivially.  The definition of $X\simeq_k \widetilde{X}$ implies that $\sigma\in X$ if and only if $\sigma\in Y$. In particular, if $\dim(X)=1$ then $X\simeq_k \widetilde{X}$ implies that the $k$-neighbourhoods of $v\in X_0$ in $X$ and $\widetilde{X}$ are isomorphic as graphs.

 A complex $X$ on $X_0$  is said to be {\it reconstructable} (up to isomorphism) from its $k$-neighbourhoods if $X \simeq_k \widetilde{X}$ implies $X \simeq \widetilde{X}$, for all complexes  $\widetilde{X}$ on $X_0$. Further, we say  $X$ is {\it exactly reconstructable} if  ${X} \simeq \widetilde{X}$ implies $X=\widetilde{X}$. We study whether the  Linial-Meshulam model is reconstructable from its $1$-neighbourhoods.  The  Linial-Meshulam model is a random complex of the form \eqref{eqn:complex}.

In the rest of the article, for $d\in \N$, the complex will be of the form 
\begin{equation}\label{eqn:complex}
	X:=\{\emptyset, X_0,X_1,\ldots,X_{d-1},X^d\}:=\l(\bigcup_{k=-1}^{d-1}X_k\r)\cup X^d,
\end{equation}
where $X_{-1}:=\emptyset$, $X_0 := \{1,2,\ldots,n\}$, 
$X_k :=\{\{i_0,\ldots,i_k\}: 1\leqslant i_0 < \cdots < i_k \leqslant n\}$, for $1\le k\le d$, and   $X^d\subseteq X_d$.  Note that  $X_k$ denotes the set of all $k$-dimensional simplexes on $X_0$. In this model, the complex contains all the simplexes up to  the  dimension $(d-1)$ and a few $d$-dimensional simplexes.

Note that if two complexes $X$ and $\widetilde{X}$ on $X_0$ are of the form \eqref{eqn:complex} then 
\[
N_{k,X}(\sigma)=N_{k,\widetilde{X}}(\sigma),
\; \mbox{ whenever $\dim(\sigma)\le d-2$}.
\]
The neighbourhoods can differ only if $\dim(\sigma)=d-1$. 
Thus, in this case, the collection of $k$-neighbourhoods of  $X$ will be denoted by
\begin{align}\label{eqn:kneighbour}
	\mathcal N_{k}(X):=\{N_{k,X}(\sigma)\suchthat \sigma\in X_{d-1}\}.	
\end{align}
We say  a complex $X$ of the form \eqref{eqn:complex} is reconstructable from its $k$-neighbourhoods  if, for all  $\widetilde{X}$ of the form \eqref{eqn:complex},
\[
X\simeq \widetilde{X} \mbox{ whenever } N_{k,X}(\sigma)\simeq N_{k,\widetilde{X}}(\sigma) \mbox{ for all } \sigma\in X_{d-1}.
\]
Similarly, we say  $X$ is exactly reconstructable from its $k$-neighbourhoods  if 
\(
X= \widetilde{X} \mbox{ whenever } N_{k,X}(\sigma)\simeq N_{k,\widetilde{X}}(\sigma) \mbox{ for all } \sigma\in X_{d-1} .
\)

%Now we define random graphs associate to the random simplicial complex $Y_{d,p_n}$. We say two simplexes $\sigma_1,\sigma_2\in X_{d-1}$ are connected, denoted by $\sigma_1\sim \sigma_2$, if $\sigma_1\cup\sigma_2\in X_{d,p_n}$. Observe that $\sigma_1\sim \sigma_2$ only if $\sigma_1\cap\sigma_2\in X_{d-2}$ for $d\ge 2$. The graph associate to the simplicial complex $Y_{d,p_n}$ is denoted by $$\GG\equiv \GG(n,p_n,d)=(\V,\EE),$$ where  $\V=X_{d-1}$ and $\EE=\EE_{d,p_n}=\{(\sigma, \sigma^\prime) \in \V\times \V: \sigma \sim \sigma^\prime\}$. Observe that, in particular, $\GG(n,p_n,1)$ gives the celebrated Erd\H os-R\'enyi graph.

%Let $G\in \GG$. The $1$-neighbourhood of a vertex $\sigma\in \V$ is denoted by $N_G(\sigma)$, where 
%\begin{align} \label{n1g.defn}
%    N_G(\sigma)= \{\sigma^\prime \in \V : \sigma \sim \sigma^\prime\}.
%\end{align}
The {\it degree} of a simplex $\sigma\in X_{d-1}$ is denoted by
$$
\deg(\sigma)=\deg_X(\sigma):= \sum_{\tau \in X^d}\one_{\{\sigma\subset \tau\}},
$$
the number of $d$-dimensional simplexes containing $\sigma$.
Observe that a $\tau\in X^d$ will contribute non-zero value in the last equation if $\tau=\sigma\cup\{v\}$ for some $v\in X_0\backslash \sigma$.   The set of neighbours of $\sigma\in X_{d-1}$ is denoted by $S_\sigma$, that is,
\[
S_\sigma=:\{\sigma'\in X_{d-1}\suchthat \sigma'\sim \sigma\}.
\]
Note that the number of elements in $S_\sigma$ is $d$ times $deg(\sigma)$, that is,
\[
|S_\sigma|=d\deg(\sigma).
\]
For any finite set $A$, the notation $|A|$ will denote the number of elements in $A$. For an example see Figure \ref{fig:degree}.
\begin{figure}[h]
	\includegraphics[scale=0.1]{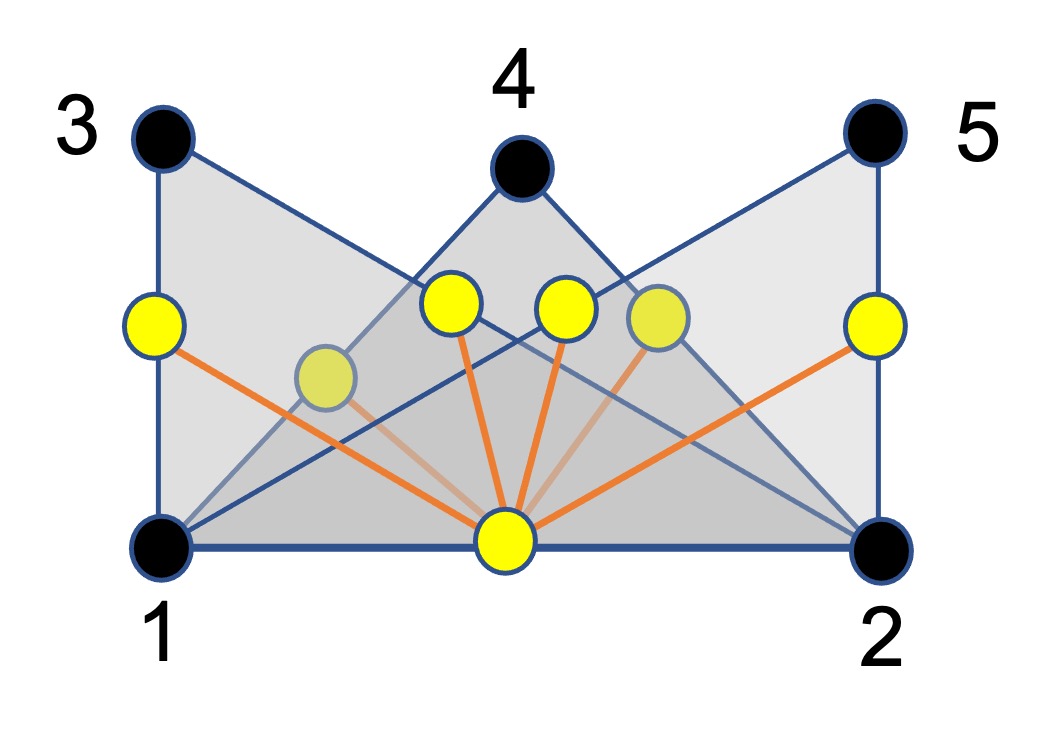}
	\caption{ Here $d=2$, $\deg_X(1,2)=3$ and $|S_\sigma|=6$}\label{fig:degree}
\end{figure}

Next we recall the Linial-Meshulam model, which is a random simplicial complex.  Let $X_{d,p_n} \left(\subseteq X_d\right)$ denote  the collection of random $d$-simplexes, where each  $\sigma\in X_d$ is chosen independently with probability $p_n$. Define a random simplicial complex 
$$Y_d(n,p_n) := \{\emptyset, X_0,\ldots,X_{d-1},X_{d,p_n}\},$$
which is known as the {\it Linial-Meshulam model}.  Observe that $Y_d(n,p_n)$ contains all the simplexes up to  the  dimension $(d-1)$, whereas each $d$-dimensional simplex  is included in the complex with probability $p_n$ and independently. 

It is easy to see that the degree of $\sigma\in X_{d-1} $ in $Y_d(n,p_n)$ is a Binomial random variable with parameters $(n-d, p_n)$, that is, $\deg_{Y_d(n,p_n)}(\sigma)\sim Bin(n-d, p_n)$. The use of the notation \enquote{$\sim$} will always be clear form the context as the same is also used for two neighbouring simplexes.

	\section{Main Results}\label{sec.mainres}
	In this section we state our main results, and give the key idea of the proofs. Let us define the high probability events.
We say	a sequence of events $A_n$ occurs \emph{with high probability} if
		\[
		P(A_n^c)=o\left( \frac{1}{n^s}\right)\,,
		\]
		for some $s>0$. We write $a_n = o(b_n)$ for two sequence of numbers $\{a_n\}_{n=1}^{\infty}$ and $\{b_n\}_{n=1}^{\infty}$ if $\left|a_n/b_n\right|\to 0$ as $n \to \infty$.

	 In \cite{gaudio2020shotgun}, it was shown that the Er\H os-R\' enyi graph $\GG(n,p_n)=Y_1(n,p_n)$ with $p_n=n^{-\alpha}$ can be exactly reconstructed from its $1$-neighbourhoods with high probability when $0< \alpha <1/3$. We extended this result for $d\in \N$.
	\begin{theorem} \label{main.thm.1}
		The Linial-Meshulam  model $Y_d(n,p_n)$ where $p_n=n^{-\alpha}$ for $0< \alpha< 1/3$, is  exactly reconstructable from its $1$-neighbourhoods with high probability. 
	\end{theorem}
 The idea of the proof of Theorem \ref{main.thm.1} is similar to the proof of  \cite[Theorem 2]{gaudio2020shotgun} (that is, the Erd\H os-R\'enyi graph $\mathcal G(n,p_n)\equiv Y_1(n,p_n)$ is reconstructable for $0<\alpha < 1/3$), but the details require some more carefulness.  We don't think the range $0<\alpha<1/3$ is optimal for the reconstruction of $Y_d(n,p_n)$. We have following conjecture.
	\begin{conjecture}\label{conj}
		The	 Linial-Meshulam model $Y_d(n,p_n)$ with  $p_n=n^{-\alpha}$ is  exactly reconstructable from its $1$-neighbourhoods with high probability if $0<\alpha <1/2$. 
	\end{conjecture}
	
	One can try to prove Conjecture \ref{conj} using the method that used in \cite{huang2021shotgun}. An other direction of work would be considering the reconstruction problem from its $2$-neighbourhoods using the method used in \cite{gaudio2020shotgun}. These remain for future works.

	  The next result is about non-constructibility of $Y_d(n,p_n)$. For $d=1$, the graph $\GG(n,p_n)\equiv Y_1(n,p_n)$ is non-reconstructible from its $1$-neighbourhoods with high probability when $1/2<\alpha<1$. We show that the same result holds for all $d\ge 1$.
	\begin{theorem}\label{main.thm.2}
	    The Linial-Meshulam graph $Y_d(n,p_n)$ where $p_n=n^{-\alpha}$ for $1/2< \alpha< 1$, cannot be reconstructed from its $1$-neighbourhoods with high probability. 
	\end{theorem}
	
\noindent Again the idea of the proof of this result is similar to the proof of  \cite[Theorem 3]{gaudio2020shotgun}, but  the calculations are more complicated.

\section{Proof of Theorem \ref{main.thm.1} }\label{sec.pfmain1}
In this section we prove Theorem \ref{main.thm.1}. The following lemmas will be used in the proof.  Throughout we use $p_n=n^{-\alpha}$, where $0<\alpha<1$.

We first state a generalization of the fingerprint lemma \cite[Lemma 2]{gaudio2020shotgun}. Let $\sigma_1, \sigma_2 \in X_{d-1}$. We say there is an edge between $\sigma_1$ and $\sigma_2$ in $X$, denoted by $(\sigma_1,\sigma_2)$, if $\sigma_1\sim \sigma_2$ in $X$. For $\sigma_1\sim \sigma_2$,  $H_{\sigma_1,\sigma_2}(X)$ denotes the sub-complex induced by  the simplexes of $S_{\sigma_1} \cap S_{\sigma_2}$, that is,
\[
H_{\sigma_1,\sigma_2}=H_{\sigma_1,\sigma_2}(X):=\{\tau\in X\suchthat \tau\subseteq \sigma\cup \sigma' \mbox{ for some $\sigma,\sigma'\in S_{\sigma_1} \cap S_{\sigma_2}$}\}.
\] 
Two edges $(\sigma_1,\sigma_2)$ and $(\sigma_3,\sigma_4)$ are said to be equal, denoted by $(\sigma_1,\sigma_2)=(\sigma_3,\sigma_4)$, if either $\sigma_1 =\sigma_3$, $\sigma_2 = \sigma_4$ or $\sigma_1 =\sigma_4$, $\sigma_2 = \sigma_3$. It is clear  that if $(\sigma_1,\sigma_2)=(\sigma_3,\sigma_4)$ then $H_{\sigma_1,\sigma_2}\simeq H_{\sigma_3,\sigma_4}$. If $\dim(X)=1$ and $v_1,v_2\in X_0$ such that $v_1\sim v_2$ then $H_{v_1,v_2}$ is the subgraph induced by the common neighbours of $v_1$ and $v_2$.
 
\begin{lemma}[Fingerprint Lemma]\label{lemma.gm.1}
	Let $X$ be a complex of the form $\{\emptyset, X_0,\ldots,X_{d-1},X^d\}$ where $X^d\subseteq X_d$.	If two edges $(\sigma_1,\sigma_2)$ and $(\sigma_3,\sigma_4)$ are equal  whenever $H_{\sigma_1,\sigma_2}$ and $H_{\sigma_3,\sigma_4}$ are  isomorphic then $X$ can be exactly reconstructed from the collection of its $1$-neighbourhoods.
\end{lemma}

In the next lemma, we give an upper bound (with high probability) on the number of simplexes that are connected with both $\sigma_1,\sigma_2\in X_{d-1}.$
\begin{lemma}\label{lem.recon.1}
	Let $\sigma_1,\sigma_2\in X_{d-1}$ such that $\sigma_1\cup \sigma_2\in X_{d,p_n}$, that is, $\sigma_1\sim \sigma_2$. The number of  simplexes that are neighbours of  $\sigma_1$ and $\sigma_2$ is denoted by $W_{\sigma_1,\sigma_2}$, that is,
	\[
	W_{\sigma_1,\sigma_2}=:|\{\sigma\in X_{d-1} \suchthat \sigma\sim \sigma_1, \sigma \sim\sigma_2 \}|.
	\]
	Then there exists a positive constant $C$ such that 
	\begin{align}\label{eq.1.0.5}
		P(W_{\sigma_1,\sigma_2} \geqslant d-1+n^c(n-d-1)p_n^2) \leqslant \exp (-Cn^{1+c-2\alpha}).
	\end{align}
	In particular, if $c>2\alpha-1$, we obtain  $W_{\sigma_1,\sigma_2}-d+1 \leqslant n^{1+c-2\alpha}$ with high probability.
\end{lemma}

In the next lemma, for $\sigma_1\sim \sigma_2$ and $\sigma_3\sim \sigma_4$, we derive a lower bound (with high probability) on the number of simplexes that are connected only with $\sigma_1,\sigma_2\in X_{d-1}$, not with $\sigma_3, \sigma_4$. We write $a_n=\Theta(b_n)$ if there exist  $C_1,C_2>0$ such that $C_1b_n\le a_n\le C_2b_n$ for all large $n$.

\begin{lemma}\label{lem.recon.2}
	Let $\sigma_1, \sigma_2,\sigma_3,\sigma_4\in X_{d-1}$ such that $\sigma_1\sim \sigma_2$ and $\sigma_3\sim \sigma_4$. Define 
	\begin{align*}
		S&=S_{\sigma_1,\sigma_2,\sigma_3,\sigma_4}:=\{\sigma\in X_{d-1} \suchthat \sigma\sim \sigma_i \mbox{ for } i=1,2,3,4\},
		\\Z&=Z_{\sigma_1,\sigma_2,\sigma_3,\sigma_4}:=\one\{\sigma_1\sim \sigma_3,\sigma_1\sim\sigma_4\}+\one\{\sigma_2\sim \sigma_3,\sigma_2\sim\sigma_4\}.
	\end{align*}
 If $(1-2\alpha)>0$ then, for large $n$,
	\begin{equation} \label{eq.1.1}
		P\left(W_{\sigma_1,\sigma_2} -|S|-Z \leqslant \frac{1}{2}np_n^2\right) \le \exp(-\Theta(n^{1-2\alpha})).
	\end{equation}
%where $|C|$ denotes the number of simplexes that are connected with all four simplexes $\sigma_1,\ldots,\sigma_4$.
\end{lemma}
\noindent The proofs of Lemmas \ref{lem.recon.1} and \ref{lem.recon.2} are given at the end of this section.
We note down \cite[Lemma $3$, Lemma $4$]{gaudio2020shotgun} which will be used in the proofs.

\begin{lemma}\label{lemma.gm.cb}[Chernoff's bound]
	Let $X_1,X_2, \ldots,X_n$ be independent indicator random variables and call $X= \sum_{i=1}^n X_i$. Then for any $\delta >0$,
	$$P(X\leqslant (1-\delta)\E(X)) \leqslant \exp\left(-\frac{\delta^2}{2}\E(X) \right) \text{ and}$$
	$$P(X\geqslant (1+\delta)\E(X)) \leqslant \exp\left(-\frac{\delta^2}{2+\delta}\E(X) \right).$$
\end{lemma}

\begin{lemma}\cite[Lemma 4]{gaudio2020shotgun}\label{lem:rec.compair}
	Let $X$ and $Y$ be random variables such that conditioned on $Y$, $X\sim Bin(Y,p)$. Let $Z(m)\sim Bin(m,p)$. Then 
	\[
	\P(X\le t_1\given Y\ge t_2)\le \P(Z(t_2)\le t_1) \mbox{ and } \P(X\ge t_2\given Y\le t_1)\le \P(Z(t_2)\ge t_1).
	\]
\end{lemma}
\noindent Now we proceed to prove Theorem \ref{main.thm.1}.
\begin{proof}[Proof of Theorem \ref{main.thm.1}] Let $\sigma_1,\sigma_2\in X_{d-1}$ such that $\sigma_1\sim \sigma_2$. Suppose  $H_{\sigma_1,\sigma_2}$ denotes the sub-complex induced by  the vertices of $S_{\sigma_1} \cap S_{\sigma_2} $ in $Y_d(n,p_n)$.  Note that the sub-complex $H_{\sigma_1,\sigma_2}$ is random.	For $\sigma_1,\sigma_2,\sigma_3,\sigma_4\in X_{d-1} $ such that $\sigma_1\sim \sigma_2$, $\sigma_3\sim \sigma_4$, we show that, for $s>0$, 
	\begin{align}\label{eqn:isomorphic}
		\P(H_{\sigma_1,\sigma_2}\simeq H_{\sigma_3,\sigma_4})=o(n^{-s}) \mbox{ whenever } (\sigma_1,\sigma_2)\neq (\sigma_3,\sigma_4).
	\end{align}
 	Then the result follows from Lemma \ref{lemma.gm.1} and \eqref{eqn:isomorphic}.  It remains to prove  \eqref{eqn:isomorphic}.

 Let $S$ be as defined in Lemma \ref{lem.recon.2} and $Y_1$ be the sub-complex induced by the simplexes of  $S_{\sigma_1} \cap S_{\sigma_2} \backslash (S\cup\{\sigma_3,\sigma_4\})$, the  shared neighbours of $\sigma_1$ and $\sigma_2$ (excluding $\sigma_3$ and $\sigma_4$) that are not neighbours of both $\sigma_3$ and $\sigma_4$. Let $Y_2$ be the sub-complex induced by the simplexes of $S_{\sigma_3} \cap S_{\sigma_4} $. Note that $Y_1$ and $Y_2$ are disjoint by construction.

Observe that if $H_{\sigma_1,\sigma_2}\simeq H_{\sigma_3,\sigma_4}$ then $W_{\sigma_1,\sigma_2}=W_{\sigma_3,\sigma_4}$ and $Y_1$ can be embedded into $Y_2$ as a sub-complex of $Y_2$ (we write $Y_1 \subset Y_2$ with the abuse of notation). Thus
\begin{align*}
\P(H_{\sigma_1,\sigma_2} \simeq H_{\sigma_3,\sigma_4})\le \P(Y_1\subset Y_2).
\end{align*}
We show that 
\begin{align}\label{eqn:graph}
P\left(Y_1 \subset Y_2\r)\le n^4(n^{an^{1+c-2\alpha}-bn^{2-5\alpha}}+exp(-Cn^{1+c-2\alpha})),
\end{align}
where $\max\{0,2\alpha-1\} <c< 1-3\alpha$ and $C>0$. The right hand side of the above equation will go to zero if $2-5\alpha > 1+c-2\alpha$, which is equivalent to say that $c<1-3\alpha$. This is a consistent condition if $\alpha<\frac{1}{3}$. 
Applying a union bound, 
\begin{align*}
	&\P\{\exists \sigma_1\sim \sigma_2,\sigma_3\sim \sigma_4 \suchthat H_{\sigma_1,\sigma_2}\simeq H_{\sigma_3,\sigma_4}\}\\&\le n^{4d}\P\{H_{\sigma_1,\sigma_2}\simeq H_{\sigma_3,\sigma_4}\}
	\\&\le n^{4(d+1)}(n^{an^{1+c-2\alpha}-bn^{2-5\alpha}}+exp(-Cn^{1+c-2\alpha}))
	\\&=o(n^{-s}),
\end{align*}
for any $s>0$ as $\alpha<1/3$. Thus, for any $(\sigma_1,\sigma_2)\neq(\sigma_3,\sigma_4)$, we have $H_{\sigma_1,\sigma_2}\not\simeq H_{\sigma_3,\sigma_4}$ with high probability if $\alpha<1/3$. This proves result.

The rest of the proof is dedicated to prove \eqref{eqn:graph}.
We have
\begin{align*}
	&\P(Y_1\subset Y_2)\\\le &\sum_{\lambda,\mu,k}\P(\left\{Y_1\subset Y_2 \suchthat W_{\sigma_1, \sigma_2} = W_{\sigma_3,\sigma_4} = \lambda+Z, |S|=\mu, Supp_d(Y_1^{d-1})=k\right\}),\nonumber
\end{align*}
where $Supp_d(A) =| \{\sigma_1 \cup \sigma_2 \in X_{d,p_n}| \sigma_1, \sigma_2 \in A\}|$ for $A \subseteq X_{d-1}$.  
%Thus,
%\begin{align*}
%    &\P(\left\{H_{\sigma_1,\sigma_2} \simeq H_{\sigma_3,\sigma_4} \suchthat W_{\sigma_1, \sigma_2} = W_{\sigma_3,\sigma_4} = \lambda+Z, |Y|=\mu, Supp_d(V(G_1))=k\right\})\nonumber \\
%    \le &P\left(G_1 \subset G_2 \suchthat W_{\sigma_1, \sigma_2} = W_{\sigma_3,\sigma_4} = \lambda+Z, |Y|=\mu, Supp_d(V(G_1))=k\right).
%\end{align*}
Note that, given $|S|=\mu$, at most $2\mu+1$ $d$-simplexes are revealed in $X_{d,p_n}$. Therefore 
\begin{align*}
&P\left(Y_1 \subset Y_2 \suchthat W_{\sigma_1, \sigma_2} = W_{\sigma_3,\sigma_4} = \lambda+Z, |S|=\mu, Supp_d(Y_1^{d-1})=k\right)\nonumber\\
\leqslant & \binom{\lambda+2}{\lambda-\mu} (\lambda -\mu)!p_n^{k-2\mu-1}
\\\leqslant& (\lambda+2)^{\lambda-\mu} \left(n^{-\alpha}\right)^{k-2\lambda-1},
\end{align*}
as  $\mu \leqslant \lambda$ and $p_n=n^{-\alpha}$. Again, $\lambda+2\le n$ and $\lambda-\mu\le \lambda$ implies that 
\begin{align}\label{eqn:upperp}
&P\left(Y_1 \subset Y_2 \suchthat W_{\sigma_1, \sigma_2} = W_{\sigma_3,\sigma_4} = \lambda+Z, |S|=\mu, Supp_d(Y_1^{d-1})=k\right)\nonumber
\\&\le \exp \left(\lambda \log (n) -\alpha(k-2\lambda-1)\log (n)\right)\nonumber
\\
&\leqslant  \exp\{((2\alpha+1)\lambda+\alpha-\alpha k)\log n\}\nonumber
\\&=n^{(2\alpha+1)\lambda+\alpha-\alpha k}.
\end{align}
 Next we complete the proof of \eqref{eqn:graph} using the following two claims
\begin{align}\label{eqn:upperlambda}
	\P(\lambda\le n^{1+c-2\alpha})&\ge 1- \exp(-Cn^{1+c-2\alpha}).\\
	\P(k\ge C n^{2-5\alpha} )&\ge 1-\exp(-C_2n^{2-5\alpha}). 	\label{eqn:lowerk}
\end{align}
Using \eqref{eqn:upperlambda} and \eqref{eqn:lowerk} from \eqref{eqn:upperp} we get 
\begin{align*}
	P\left(Y_1 \subset Y_2 \suchthat W_{\sigma_1, \sigma_2} = W_{\sigma_3,\sigma_4},  |S|, Supp_d(Y_1^{d-1})\right)\le n^{an^{1+c-2\alpha}-bn^{2-5\alpha}},
\end{align*}
with probability at least $1-\exp(-Cn^{1+c-2\alpha})$. Therefore  we get 
\[
P\left(Y_1 \subset Y_2\r)\le n^4.n^{an^{1+c-2\alpha}-bn^{2-5\alpha}}+n^4exp(-Cn^{1+c-2\alpha}),
\]
as $\lambda,\mu\le n$ and $k\le n^2$. This completes the proof of \eqref{eqn:graph}. It remains to prove \eqref{eqn:upperlambda} and \eqref{eqn:lowerk}.

\vspace{.2cm}
\noindent{\it Proof of \eqref{eqn:upperlambda}:}  Observe that \eqref{eqn:upperlambda} follows from 
Lemma \ref{lem.recon.1}.

\vspace{.2cm} 
\noindent {\it Proof of \eqref{eqn:lowerk}:} 
Clearly, given $W_{\sigma, \sigma^\prime} -|S|-Z$, $Supp_d(Y_1^{d-1})\sim Bin(W_{\sigma, \sigma^\prime} -|S|-Z,p_n)$. From Lemma \ref{lem.recon.2}, we have 
\[
\P(W_{\sigma, \sigma^\prime} -|S|-Z\ge \frac{1}{2}np_n^2)\ge 1- e^{-\Theta(n^{1-2\alpha})}.
\]
The right hand side goes to $1$ if $1-2\alpha>0$. Lemma \ref{lem:rec.compair} and Lemma \ref{lemma.gm.cb} imply that
\begin{align*}
    &P\left(Supp_d(Y_1^{d-1}) \leqslant (1-\epsilon)p_n \begin{pmatrix} \frac{1}{2}np_n^2 \\ 2 \end{pmatrix}\given (W_{\sigma, \sigma^\prime} -|S|-Z)\ge\frac{1}{2}np_n^2 \right) \nonumber 
    \\&\le P\left(Bin(\frac{1}{2}np_n^2,p_n)\leqslant (1-\epsilon)p_n \begin{pmatrix} \frac{1}{2}np_n^2 \\ 2 \end{pmatrix}\r)\nonumber
    \\&\leqslant \exp\left(-\frac{\epsilon^2}{2}p_n\begin{pmatrix} \frac{1}{2}np_n^2 \\ 2 \end{pmatrix} \right)\nonumber \\
    &= \exp(-C_2n^{2-5\alpha}),
\end{align*}
for some positive constant $C_2$. Thus we have 
\begin{align*}
	\P(Supp_d(Y_1^{d-1})\ge C n^{2-5\alpha} )\ge 1-\exp(-C_2n^{2-5\alpha}),
\end{align*}
for some positive constant $C$. This complete the proof.
\end{proof}

Next we give the proofs of Lemmas \ref{lemma.gm.1}, \ref{lem.recon.1} and \ref{lem.recon.2}. The proof of Lemma \ref{lemma.gm.1}  can be derived from \cite[Lemma 2]{gaudio2020shotgun}, for sake of completeness we give a proof.
\begin{proof}[Proof of Lemma \ref{lemma.gm.1}]
	Since $X$ has complete $(d-1)$-dimensional skeleton, in order to reconstruct $X$ it is enough to check whether any two simplexes $\sigma_1, \sigma_2 \in X_{d-1}$ ($\sigma_1 \neq \sigma_2$) are connected in $X$. To determine that, we examine the neighbourhoods of $\sigma_1, \sigma_2$ by observing the sub-complexes $H_{\sigma_1,\sigma_3}$ and $H_{\sigma_2,\sigma_4}$ for neighbours $\sigma_1 \sim \sigma_3$ and $\sigma_2 \sim \sigma_4$. The reconstruction algorithm is as follows: We conclude that $\sigma_1 \sim \sigma_2$ in $X$ if there exist $\sigma_3, \sigma_4 \in X_{d-1}$ such that $\sigma_1 \sim \sigma_3$, $\sigma_2 \sim \sigma_4$ and $H_{\sigma_1,\sigma_3}$ is isomorphic with $H_{\sigma_2,\sigma_4}$.

	Suppose $\sigma_1 \sim \sigma_2$ in $X$. We choose $\sigma_3 = \sigma_1$ and $\sigma_4 = \sigma_2$. Then, $H_{\sigma_1,\sigma_3} = H_{\sigma_2,\sigma_4}= H_{\sigma_1,\sigma_2}$. Conversely, suppose there are some  $\sigma_3, \sigma_4 \in X_{d-1}$ such that $\sigma_1 \sim \sigma_3$, $\sigma_2 \sim \sigma_4$ and $H_{\sigma_1,\sigma_3}$ is isomorphic with $H_{\sigma_2,\sigma_4}$. Then, the hypothesis of the lemma says that $(\sigma_1,\sigma_3) = (\sigma_2,\sigma_4)$. Therefore, $\sigma_1=\sigma_4$ and $\sigma_3=\sigma_2$ because $\sigma_1\neq \sigma_2$. Hence, $(\sigma_1,\sigma_2)$ is an edge in $X$ or in other words $\sigma_1\sim \sigma_2$.

	So, continuing the process described in the algorithm, we recover the complex.
\end{proof}

 \begin{proof}[Proof of Lemma \ref{lem.recon.1}]
 	For $\sigma_1\sim \sigma_2$, define $S_{\sigma_1,\sigma_2}=S_{\sigma_1,\sigma_2}'\cup S_{\sigma_1,\sigma_2}''$ where
 	\begin{align*}
 		S_{\sigma_1,\sigma_2}'&=\{\sigma\in X_{d-1}\suchthat \sigma\sim \sigma_1,\sigma\sim \sigma_2, \sigma\subset \sigma_1\cup \sigma_2\}\\
 		S_{\sigma_1,\sigma_2}''&=\{\sigma\in X_{d-1}\suchthat \sigma\sim \sigma_1,\sigma\sim \sigma_2, \sigma\nsubseteq \sigma_1\cup \sigma_2\}
 	\end{align*}
 Clearly, $W_{\sigma_1,\sigma_2}=|S_{\sigma_1,\sigma_2}'|+|S_{\sigma_1,\sigma_2}''|$. 
 	Observe that if $\sigma_1,\sigma_2\neq \sigma\in X_{d-1}$ such that $\sigma\subset \sigma_1\cup\sigma_2$ then $\sigma\sim \sigma_1$ and $\sigma\sim \sigma_2$. Therefore 
 \begin{align}\label{eqn:s'}
  |S_{\sigma_1,\sigma_2}'|=d-1.
 \end{align}
 	Again  $\sigma_1\sim \sigma_2$ implies that $\sigma_1\cap \sigma_2\in X_{d-2}$. Therefore if $\sigma\sim \sigma_1,\sigma_2$  but $\sigma\nsubseteq \sigma_1\cup\sigma_2$ then $\sigma$ will be of the form $(\sigma_1\cap\sigma_2)\cup\{v\}$ for some $v\in X_0\backslash (\sigma_1\cup\sigma_2)$. See Figure \ref{fig:wsigma}.

 \begin{figure}[h]
 		\includegraphics[scale=0.1]{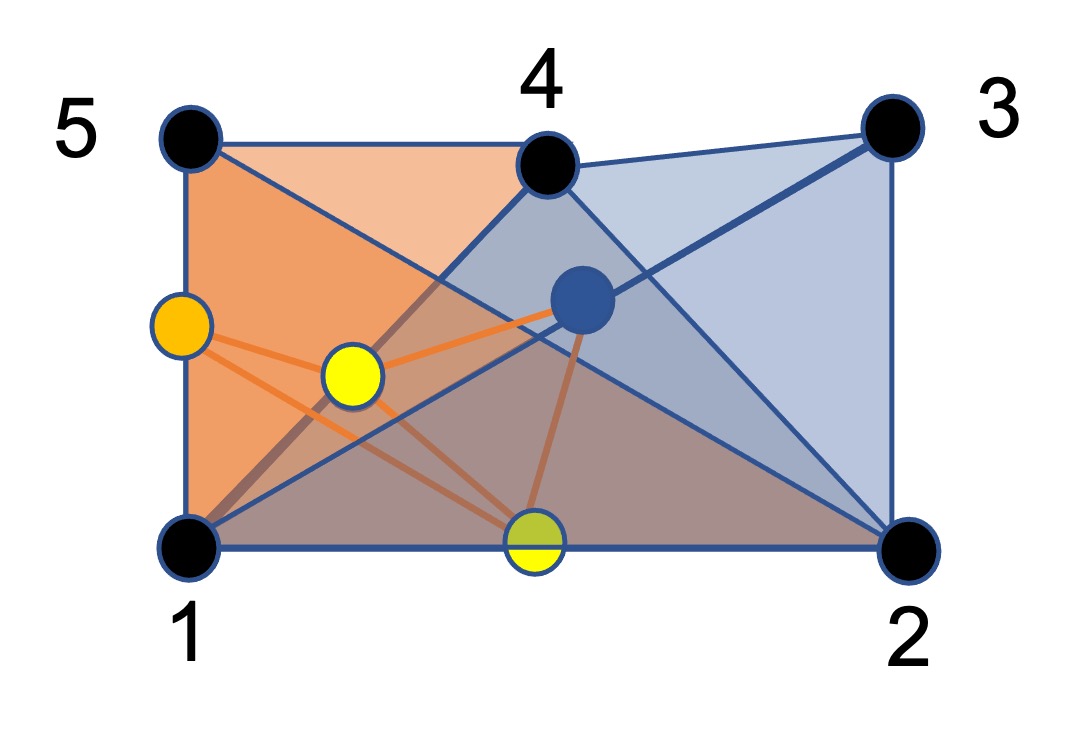}
 		\caption{Simplexes  $(1,5)$ and $(1,3)$ are connected with both $(1,4), (1,2)$, as $(125), (145), (134), (123)$ simplexes are included in the complex.}\label{fig:wsigma}
 	\end{figure}
 
 \noindent 	Therefore $S_{\sigma_1,\sigma_2}''$ can be written as 
 	\[
 		S_{\sigma_1,\sigma_2}''=\{(\sigma_1\cap\sigma_2)\cup\{v\}\suchthat v\in X_0, (\sigma_1\cap\sigma_2)\cup\{v\}\in X_{d,p}\}.
 		\]
	Which implies that $|S_{\sigma_1,\sigma_2}''|\sim Bin(n-d-1, p_n^2)$, as we need to add two simplexes to get an element in $S_{\sigma_1,\sigma_2}''$. See Figure \ref{fig:wsigma}. Let $c>0$, to be determined later. Using Lemma \ref{lemma.gm.cb}, we have 
	\begin{align}\label{eqn:s''}
		P(| S_{\sigma_1,\sigma_2}''|\geqslant n^c(n-d-1)p_n^2) &\leqslant \exp\left(-\frac{n^{2c}}{1+n^c}(n-d-1)n^{-2\alpha}\right) \nonumber\\
		& \leqslant \exp (-Cn^{1+c-2\alpha}),
	\end{align}
	for some constant $C>0$. We get the result by combining \eqref{eqn:s'} and \eqref{eqn:s''}. Clearly,  the right hand side goes to zero if $c>2\alpha-1$.
\end{proof}

\begin{proof}[Proof of Lemma \ref{lem.recon.2}]
%We show that $W_{\sigma_1,\sigma_2}-|Y|-Z$ is a binomial random variable. However, the parameters depend on the choices of $\sigma_1,\sigma_2,\sigma_3,\sigma_4$.  
Let $\sigma\in S$. Then $\sigma\sim \sigma_1,\sigma_2$ and $\sigma\sim \sigma_3,\sigma_4$ which imply that
\begin{align*}
	\sigma=(\sigma_1\cap \sigma_2)\cup\{v\} \mbox{ and } \sigma=(\sigma_3\cap \sigma_4)\cup\{v'\}
\end{align*} 
for some $v, v'\in X_0$. Thus we get the following identity
\begin{align}\label{eqn:identity}
(\sigma_1\cap \sigma_2)\cup\{v\}=(\sigma_3\cap \sigma_4)\cup\{v'\}.
\end{align}

\noindent{\bf Case-I:} Suppose  $|\cap_{i=1}^4\sigma_i|=d-1$, that is, $(\sigma_1\cap \sigma_2)=(\sigma_3\cap \sigma_4)$. Then any $v=v'\in X_0\backslash  (\cup_{i=1}^4\sigma_i)$ satisfies \eqref{eqn:identity}. Thus 
\[
\P((\sigma_1\cap \sigma_2)\cup\{v\}\in S_{\sigma_1,\sigma_2}''\backslash S)=p_n^2(1-p_n^2),
\]
where $S$ is as defined in Lemma \ref{lem.recon.2}. Therefore we get  $|S_{\sigma_1,\sigma_2}''|-|S|-Z\sim Bin(n-d-3, p_n^2(1-p_n^2))$. Fix $0<\eps<1/2$. Lemma \ref{lemma.gm.cb} implies that 
\[
\P(|S_{\sigma_1,\sigma_2}''|-|S|-Z\le \eps (n-d-3)p_n^2(1-p_n^2))\le \exp(-\frac{\eps^2}{2}(n-d-3)p_n^2(1-p_n^2)).
\]
If $1-2\alpha>0$, then the last equation implies that 
\begin{align}
\P(W_{\sigma_1,\sigma_2}-|S|-Z\le \frac{1}{2} np_n^2)\le \exp(-\Theta(n^{1-2\alpha})),
\end{align}
as $|S_{\sigma_1,\sigma_2}'|=d-1$.

\vspace{.2cm}
\noindent{\bf Case-II:}   Suppose $|\cap_{i=1}^4\sigma_i|=d-2$. Then there is only one choice of $v,v'$ which satisfies \eqref{eqn:identity}, namely, $v=(\sigma_1\cap\sigma_2)\backslash (\sigma_3\cap \sigma_4)$ and $v'=(\sigma_3\cap\sigma_4)\backslash (\sigma_1\cap \sigma_2)$. Thus $|S|\le 1$.   We have 
\[
W_{\sigma_1,\sigma_2}-|S|-Z\le W_{\sigma_1,\sigma_2}.
\]
Next we give bound on $W_{\sigma_1,\sigma_2}$. We have 
\[
W_{\sigma_1,\sigma_2}=|S_{\sigma_1,\sigma_2}'|+|S_{\sigma_1,\sigma_2}''|,
\]
where $S_{\sigma_1,\sigma_2}'$ and $S_{\sigma_1,\sigma_2}''$ are as defined in the proof of Lemma \ref{lem.recon.1}. Since $|S_{\sigma_1,\sigma_2}''|\sim Bin(n-d-1, p_n^2)$, by Chernoff's bound (Lemma \ref{lemma.gm.cb}), for $0<\eps<1/2$,
\[
\P(|S_{\sigma_1,\sigma_2}''|\le \eps (n-d-1)p_n^2)\le \exp(-\frac{\eps^2}{2}(n-d-1)p_n^2).
\]
If $1-2\alpha>0$ then the last equation implies that 
\[
\P(W_{\sigma_1,\sigma_2}\le \frac{1}{2}np_n^2)\le \exp(-\Theta(n^{1-2\alpha})),
\]
as $|S_{\sigma_1,\sigma_2}'|=d-1, |S|\le 1, |Z|\le 2$. Thus we get, for large $n$, 
	\begin{equation}
		P\left(W_{\sigma_1, \sigma_2} -|S|-Z \leqslant \frac{1}{2}np_n^2\right) \le \exp(-\Theta(n^{1-2\alpha})).
	\end{equation}

	\vspace{.2cm}
	\noindent{\bf Case-III:}  Suppose $|\cap_{i=1}^4\sigma_i|\le d-3$. Then there is no $v,v'\in X_0$ which satisfies \eqref{eqn:identity}. Thus $|S|=0$.  Hence the result follows as in Case II. 
	
	\vspace{.1cm}
	Similar analysis can be done when the edges are of the form $(\sigma_1,\sigma_2)$ and $(\sigma_2,\sigma_3)$. It can be shown  that $|S_{\sigma_1,\sigma_2}''|-|S|-Z\sim Bin(n-d-3, p_n^2(1-p_n))$ if $|\cap_{i=1}^3\sigma_i|= d-1$. Otherwise, $|S|\le 1$. Thus, following the calculation as in Case-I,II, if $1-2\alpha>0$ we get
	\begin{equation*}
		P\left(W_{\sigma_1, \sigma_2} -|S|-Z \leqslant \frac{1}{2}np_n^2\right) \le \exp(-\Theta(n^{1-2\alpha})).
	\end{equation*}
	We skip the details here. Hence the result.
\end{proof}

\section{Proof of Theorem \ref{main.thm.2}}\label{sec.pfmain2}
This section is dedicated for the proof of Theorem \ref{main.thm.2}. Let $X$ be a complex, where
\[
X:=\{\emptyset, X_0,\ldots,X_{d-1},X^d\},
\]
and $X^d\subseteq X_d$. Recall $S_\sigma$ denotes the the set of neighbours of $\sigma\in X_{d-1}$.
 Define 
\begin{align*}
	D_\sigma=D_{\sigma}(S_\sigma):=\{\sigma\cup \sigma' \suchthat \sigma'\in S_\sigma\} \mbox{ and } Supp_d(S_\sigma):=\{\sigma_1\cup\sigma_2\suchthat \sigma_1,  \sigma_2\in S_\sigma\}.
\end{align*}
Observe that each simplex in $D_\sigma(S_\sigma)$ contains $\sigma$. However, not every simplex  in $Supp_d(S_\sigma)$ contains $\sigma$. The set the simplexes which do not contain $\sigma$ is denoted by
\[
D_{\sigma}^*=D_{\sigma}^*(S_\sigma):=Supp_d(S_\sigma) \backslash D_{\sigma}(S_\sigma).
\]
 Observe that $|D_\sigma|=\deg(\sigma)$. The $1$-neighbourhood  of $\sigma$ in $X$ can be written as
\[
N_{1,X}(\sigma)=(S_\sigma, D_\sigma, D_\sigma^*).
\]
Fix $0<\epsilon < 1$ and $q_n = (1+\epsilon)p_n$. For $c>0$, suppose $t_n=(1+n^c)p_n$. Define
\begin{align}\label{eqn:S}
	\mathcal S=\{\{N_{1,X}(\sigma), \sigma\in X_{d-1}\} \suchthat |D_\sigma|< nq_n, |D_{\sigma}^*|< \frac{d}{2}n^2q_n^2t_n, \forall \sigma\in X_{d-1}\},
\end{align}
the set of all possible $1$-neighbourhood collections where the degree of each central vertex is less than $nq_n$ and each neighbourhood has fewer than $\frac{d}{2}n^{2}q_n^{2}t_n$ neighbouring $d$-faces those are not counted in $\deg(\sigma)$. 

%For $G\in \GG$, the $1$-neighbourhood of $G$ is denoted by $\mathcal N_1(G)$, defined as
%\[
%\mathcal N_1(G):=\{N_1{(\sigma,G)},\sigma\in X_{d-1}\},
%\]
%where $N_1(\sigma,G)$ denotes the $1$-neighbourhood of $\sigma$ in $G$.
\begin{lemma} \label{lem.thm2.1}
	Let $\mathcal N_1(X)$ and $\mathcal S$ be as defined in \eqref{eqn:kneighbour} and \eqref{eqn:S} respectively. Let $X\in Y_d(n,p_n)$ with $p_n=n^{-\alpha}$ and $0<\alpha<1$.  If  $c>3\alpha-2$ then  
	\[
	\P(\mathcal N_1(X)\in \mathcal S)\ge 1- e^{-an^{b}}
	\]
	for some positive constants $a$ and $b$.
\end{lemma}

\begin{lemma}\label{lem.thm2.2}
	Let $I:=\left\{m \in \mathbb{N}: \left|m- \binom{n}{d+1}p_n\right| < \epsilon \binom{n}{d+1} p_n\right\}$, where $p_n=n^{-\alpha}$. Then
	\[
	\P(|X_{d,p_n}|\in I)\ge 1-e^{-n^d},
	\]
where $|X_{d,p_n}|$ denotes the number of $d$-simplexes in $Y_d(n,p_n)$. 
\end{lemma}

\begin{lemma}\label{lem.thm2.3}
	Let  $\mathcal S$ be as defined in \eqref{eqn:S}.  If  $c<2\alpha-1$ then 
	\[
	\max_{m \in I}\frac{n^d|\mathcal S|}{\binom{\binom{n}{d+1}}{m}}=o(e^{-n^s}),
	\]
	for some $s>0$. 
\end{lemma}
\begin{proof}[Proof of Theorem \ref{main.thm.2}]
     We consider a particular complex reconstruction algorithm which outputs a complex when a collection of $1$-neighbourhoods is given. Let $\mathcal S$ be a collection of $1$ neighbourhoods  as defined in \eqref{eqn:S}. The algorithm maps each element of $\mathcal S$ to an isomorphism class, which corresponds to at most $n^d!$ labelled complexes. The algorithm fails if $X\in Y_d(n,p_n)$ such that $\mathcal N_1(X)\in \mathcal S$ but the output of the algorithm of $\mathcal N_1(X)$ is not isomorphic to $X$.

     %Then, for $X\in Y_d(n,p_n)$, Lemma \ref{lem.thm2.1} implies $\mathcal{N}_1(X) \in \mathcal S$ with high probability. 
    
     We condition on the event $|X_{d,p_n}|=m$ for some $m\in \mathbb{N}$. Given this information, there are 
    $\binom{\binom{n}{d+1}}{m}$ 
    possible labelled $d$-simplexes which may be chosen with equal probability.  Therefore, conditioned on $|X_{d,p_n}|=m$, the algorithm fails when any complex $X$ is not achieved by the algorithm output. Let $p_m$ denote the probability of failure given $|X_{d,p_n}|=m$. Thus, 
    \begin{align*}
    	p_m&=\P(\mbox{ algorithm fails}\given |X_{d,p_n}|=m)
        \\&\ge P\left(\mathcal{N}_1(X) \in \mathcal S | |X_{d,p_n}|=m \right) - \frac{n^d!|\mathcal S|}{\binom{\binom{n}{d+1}}{m}} 
        \\&= \frac{\binom{\binom{n}{d+1}}{m} - n^d!|\mathcal S|}{\binom{\binom{n}{d+1}}{m}} - P\left(\mathcal{N}_1(X) \notin \mathcal S | |X_{d,p_n}|=m \right).
    \end{align*}
    Let $I \subseteq \left\{1,2, \ldots,\binom{n}{d+1}\right\}$ be as defined in Lemma \ref{lem.thm2.2}. Let $p^*$ denote the overall failure probability of the algorithm. Then  
    \begin{align*}
    	p^*&\ge \sum_{m\in I}p_m
        \\&\ge \sum_{m \in I} \frac{\binom{\binom{n}{d+1}}{m}- n^d!|\mathcal S|}{\binom{\binom{n}{d+1}}{m}} P(|X_{d,p_n}|=m) - P\left(\{\mathcal{N}_1(X) \notin \mathcal S\} \cap \{|X_{d,p_n}| \in I\} \right)\\
        & \geqslant P(|X_{d,p_n}| \in I) \min_{m \in I} \frac{\binom{\binom{n}{d+1}}{m}- n^d!|\mathcal S|}{\binom{\binom{n}{d+1}}{m}} - P(\mathcal{N}_1(X) \notin \mathcal S).
    \end{align*}
    Therefore Lemmas \ref{lem.thm2.1}, \ref{lem.thm2.2} and \ref{lem.thm2.3} implies that 
    \[
    p^*\ge (1-e^{-n^d})(1-e^{-n^s})-e^{-an^b}\ge 1-e^{-a'n^{b'}},
    \]
    for some $s,a,b,a',b'>0$, if  the constant $c$ satisfies
    $$\max\{0,3\alpha-2\} < c < \min\{\alpha,2\alpha-1\}.$$
    The above is satisfied when $1/2< \alpha <1$ as required. The proof is then completed by choosing $c=1/2(\max{0,3\alpha-2}+2\alpha-1)$ since $\min\{\alpha,2\alpha-1\} = 2\alpha-1$.
\end{proof}

    The rest of the section is dedicated to prove Lemmas  \ref{lem.thm2.1}, \ref{lem.thm2.2} and \ref{lem.thm2.3}.

\begin{proof}[Proof of Lemma \ref{lem.thm2.1}]
        Let $S_\sigma$ denote the set of neighbours of $\sigma\in X_{d-1}$ in $Y_d(n,p_n)$. Consequently we define $D_{\sigma}$ and $D_\sigma^*$ as defined above. Note that $S_\sigma$ is a random set, hence $D_{\sigma}$ and $D_\sigma^*$ are random.
       We show that if $c>3\alpha -2$ then
        \begin{align}\label{eqn:lem6}
        P\left(\bigcap_{\sigma \in X_{d-1}} \Big(\{|D_{\sigma}| < nq_n \}\cap\{|D_{\sigma}^*| < \frac{d}{2}n^2q_n^2t_n \}\Big)\right)\ge 1-e^{-an^b},
        \end{align}
    for some positive constants $a$ and $b$. Clearly \eqref{eqn:lem6} gives the result.
    
    \vspace{.2cm}
    \noindent {\it Proof of \eqref{eqn:lem6}:} Observe that $|D_{\sigma}|=\deg(\sigma)\sim Bin(n-d,p_n)$. Then Lemma \ref{lemma.gm.cb} gives 
    \begin{align}\label{pf.lem.1.eq2}
    	P\left(|D_{\sigma}| < nq_n\right) &\geqslant P\left( |D_{\sigma}| < (n-d)q_n \right)\nonumber\\
    	&=1- P\left( |D_{\sigma}|\geqslant (1+\epsilon)(n-d)q_n \right)\nonumber\\
    	&\geqslant 1-\exp\left(-\frac{\epsilon^2p_n(n-d)}{3} \right).
    \end{align}
     Next we derive a bound of    $|D_{\sigma}^*|$ with high probability.  We have
        \begin{align}\label{pf.lem.1.eq1}
            &P\left(|D_{\sigma}^*| \geqslant \frac{1}{2}n^{2}q_n^{2}t_n)  \right)\nonumber 
            \\& \leqslant P\left(|D_{\sigma}^*| \geqslant \frac{1}{2}n^{2}q_n^{2}t_n \Big| |D_{\sigma}| < nq_n)  \right) + P(|D_{\sigma}| \geqslant nq_n).
        \end{align}
        Note that, conditioned on $|D_{\sigma}|$, $|D_{\sigma}^*|\sim dBin(\binom{|D_{\sigma}|}{2}, p_n)$.  Lemmas \ref{lemma.gm.cb} and \ref{lem:rec.compair} imply
        \begin{align}\label{pf.lem.1.eq3}
            P\left(|D_{\sigma}^*|\geqslant \frac{d}{2}n^{2}q_n^{2}t_n| |D_{\sigma}| \le nq_n)  \right) &\leqslant P\left(|D_{\sigma}^*|\geqslant \frac{d}{2}n^{2}q_n^{2}t_n| |D_{\sigma}| = nq_n)  \right) \nonumber\\
            &\leqslant \exp\left(-\frac{n^{2c}d}{2+n^c} \binom{nq_n}{2}p_n \right).
        \end{align}
        Combining the bounds from \eqref{pf.lem.1.eq2} and \eqref{pf.lem.1.eq3} and substituting in \eqref{pf.lem.1.eq1}, we obtain
        $$P\left(|D_{\sigma}^*|\geqslant \frac{d}{2}n^2q_n^2t_n)  \right) \leqslant \exp\left(-\frac{n^{2c}d}{2+n^c} \binom{nq_n}{d+1}p_n \right) + \exp\left(-\frac{\epsilon^2p_n(n-d)}{3} \right) .$$
        By the union bound, the required probability is given by
        \begin{align*}
            &P\left(\bigcap_{\sigma \in X_{d-1}} \Big(\{|D_{\sigma}| < nq_n \}\cap\{|D_{\sigma}^*|< \frac{d}{2}n^2q_n^2t_n \}\Big)\right) \\
            &\geqslant 1 - \binom{n}{d} \exp\left(-\frac{n^{2c}d}{2+n^c} \binom{nq_n}{2}p_n \right) + \binom{n}{d} \exp\left(-\frac{\epsilon^2p_n(n-1)}{3} \right)\\
            &\geqslant 1-n^d \exp\left(-C_3n^{2+c-3\alpha} \right) -n^d\exp\left(-C_4n^{1-\alpha} \right),
        \end{align*}
        for some positive constants $C_3$, $C_4$. The above will give us a high probability bound when $2+c-3\alpha >0$. Hence the result  if $c>3\alpha -2$, as $\alpha < 1$,  .
    \end{proof}
    
     \begin{proof}[Proof of Lemma \ref{lem.thm2.2}]
        Note that $|X_{d,p_n}\sim Bin(\binom{n}{d+1}, p_n)|$. Then Lemma \ref{lemma.gm.cb} gives
        $$P(|X_{d,p_n}| \notin I) =P\left( \left||X_{d,p_n}|- \binom{n}{d+1}p_n\right| \geqslant \epsilon \binom{n}{d+1}p_n \right) \leqslant 2\exp \left( -\frac{\epsilon^2}{3}\binom{n}{d+1}p_n\right).$$
        Using $p_n=n^{-\alpha}$ and $\binom{n}{d+1}\le n^{d+1}$, we get the result.
    \end{proof}

\begin{proof}[Proof of Lemma \ref{lem.thm2.3}]
    By the definition  of $\mathcal S$, we have $|D_\sigma|\in \{1,\ldots,q_nn\}$ and $|D_{\sigma}^*|\in \{1,\ldots,\frac{d}{2}n^2q_n^2t_n \}$ for all $\sigma\in X_{d-1}$. Again the choices for the number of neighbouring $d$-simplexes is upper bounded by $\binom{d\binom{nq_n}{2}}{\frac{d}{2}n^2q_n^2t_n}$ for each $\sigma$.  Therefore
    \begin{align}\label{pf.thm.2.eq4}
        |\mathcal S| &\leqslant \left(nq_n \cdot \frac{d}{2}n^2q_n^2t_n \cdot \binom{d\binom{nq_n}{2}}{\frac{d}{2}n^2q_n^2t_n} \right)^{\binom{n}{d}} 
        \leqslant \left( \frac{d}{2}n^3q_n^3t_n \left( \frac{e}{t_n}\right)^{\frac{d}{2}n^2q_n^2t_n} \right)^{\binom{n}{d}},
    \end{align}
in the last inequality we use the fact that $\binom{n}{k} \leqslant \left(\frac{\e n}{k} \right)^k$. 
%We have,
%        $$\min_{m \in \bar{E}} \frac{\binom{\binom{n}{d+1}}{m}- n^d!|S|}{\binom{\binom{n}{d+1}}{m}} =1-\max_{m \in \bar{E}} \frac{n^d!|S|}{\binom{\binom{n}{d+1}}{m}}.$$
        Now,
        \begin{align}\label{pf.thm.2.eq5}
            \min_{m \in I}\binom{\binom{n}{d+1}}{m}= \binom{\binom{n}{d+1}}{(1-\eps)\binom{n}{d+1}p_n}  \geqslant \left( \frac{1}{(1-\epsilon)p_n} \right)^{(1-\epsilon)\binom{n}{d+1}p_n}. 
        \end{align}
        Therefore, using \eqref{pf.thm.2.eq4}, \eqref{pf.thm.2.eq5} and the fact that $n^d!\leqslant \exp\left(dn^d\log(n)\right)$, we get
        \begin{align}\label{eqn:boundupper}
            \max_{m \in I} \frac{n^d!|S|}{\binom{\binom{n}{d+1}}{m}} &\leqslant \frac{\exp\left(dn^d\log(n)\right)\left( \frac{d}{2}n^3q_n^3t_n \left( \frac{e}{t_n}\right)^{\frac{d}{2}n^2q_n^2t_n} \right)^{\binom{n}{d}}}{\left( \frac{1}{(1-\epsilon)p_n} \right)^{(1-\epsilon)\binom{n}{d+1}p_n}}\nonumber\\
            &=\exp\bigg\{dn^d\log(n) +\binom{n}{d}\log(\frac{d}{2}n^3q_n^3t_n) + \binom{n}{d}\frac{d}{2}n^2q_n^2t_n \log \left( \frac{e}{t_n}\right) ) \nonumber
            \\ &\hspace{2cm}- (1-\epsilon)\binom{n}{d+1} p_n \log \left( \frac{1}{(1-\epsilon)p_n} \right) \bigg\} \nonumber \\
            &\le \exp\bigg\{dn^d\log(n) + C_5n^d\log (n^{3+c-4\alpha}) +C_6 n^{c-3\alpha+d+2}\log (n) \\&\hspace{2cm}-C_7 n^{1+d-\alpha} \log(n) \bigg\}\nonumber,
        \end{align}
       for some positive constants $C_5,C_6,C_7$. The right hand side of \eqref{eqn:boundupper} goes to zero exponentially when $d< d+1-\alpha$ (which is always true since $\alpha<1$) and $c-3\alpha+d+2 < 1+d-\alpha$. In other words, the required condition  is $c<2\alpha-1$. 
    \end{proof}

%Multi-parameter random simplicial complex references= \cite{CF2016book, FCF2019}
	
	\section*{Acknowledgement} The research of KA was partially supported by the Inspire Faculty Fellowship: DST/INSPIRE/04/2020/000579. SC's research was supported by the NBHM postdoctoral fellowship (order no. 0204/19/2021/R\&D-II/11871).

	\bibliographystyle{abbrv}
	\bibliography{practicebib}
\end{document}

%% file: mypreamble.tex
\theoremstyle{theorem}
    \newtheorem{theorem}{Theorem}
    \newtheorem{lemma}{Lemma}

    \newtheorem{conjecture}{Conjecture}
\theoremstyle{definition} % For roman text in the body

    \newtheorem{remark}{Remark}
    \newtheorem{example}[theorem]{Example}
    \newtheorem{exercise}[theorem]{Exercise}

\def\suchthat{\; : \;}

\def\iff{\Longleftrightarrow}

\def\e{\epsilon}

\def\GG{{\mathcal G}}

\def\Z{\mathbb{Z}}

\def\N{\mathbb{N}}

\def\Z{\mathbb{Z}}

\def\l{\left}
\def\r{\right}
\def\<{\langle}
\def\>{\rangle}

\newcommand{\one}{{\mathbf 1}}
\newcommand{\E}{\mbox{\bf E}}

\def\P{{\bf P}}

\def\given{\left.\vphantom{\hbox{\Large (}}\right|}

\newcommand\mnote[1]{} %off
\newcommand\be{\begin{equation*}}

\newcommand\ee{\end{equation*}}

\newcommand\ben{\begin{equation}}
\newcommand\een{\end{equation}}
\newcommand\bes{\begin{eqnarray*}}
\newcommand\ees{\end{eqnarray*}}

\newcommand\bex{\begin{exercise}}
\newcommand\eex{\end{exercise}}
\newcommand\beg{\begin{example}}
\newcommand\eeg{\end{example}}
\newcommand\benu{\begin{enumerate}}
\newcommand\eenu{\end{enumerate}}
\newcommand\beit{\begin{itemize}}
\newcommand\eeit{\end{itemize}}
\newcommand\berk{\begin{remark}}
\newcommand\eerk{\end{remark}}
\newcommand\bdefn{\begin{defintion}}
\newcommand\edefn{\end{definition}}
\newcommand\bthm{\begin{theorem}}
\newcommand\ethm{\end{theorem}}
\newcommand\bprf{\begin{proof}}
\newcommand\eprf{\end{proof}}
\newcommand\blem{\begin{lemma}}
\newcommand\elem{\end{lemma}}

\newcommand{\dist}{\mbox{\rm dist}}

\newcommand{\sm}{{\raise0.3ex\hbox{$\scriptstyle \setminus$}}}

\def\l{\left}
\def\r{\right}

\def\eps{\epsilon}

\def\CHI{\mathchoice%
{\raise2pt\hbox{$\chi$}}%
{\raise2pt\hbox{$\chi$}}%
{\raise1.3pt\hbox{$\scriptstyle\chi$}}%
{\raise0.8pt\hbox{$\scriptscriptstyle\chi$}}}
\def\smalloplus{\raise1pt\hbox{$\,\scriptstyle \oplus\;$}}